\documentclass[12pt]{article}

\usepackage{amsmath}
\usepackage{amsthm}
\usepackage{amsfonts}
\usepackage{ amssymb }
\usepackage{authblk}

\usepackage[english]{babel}

\usepackage{bbm} 

\usepackage{centernot} 

\usepackage{esvect} 
\usepackage{enumitem} 

\usepackage{graphicx} 

\usepackage{mathtools} 

\usepackage{physics} 

\usepackage[colorinlistoftodos]{todonotes} 
\usepackage{xcolor} 

\usepackage[utf8]{inputenc}
\usepackage[T1]{fontenc}
\usepackage{CormorantGaramond} 

\theoremstyle{plain} 
\usepackage{geometry}
\author{\sc Peter Sarnak, Nina Zubrilina}

\DeclarePairedDelimiter\floor{\lfloor}{\rfloor}

\begin{document}

\newtheorem{thm}{Theorem}
\newtheorem*{est}{Theorem}
\newtheorem*{ptsn}{Petersson Formula}
\newtheorem*{sv}{Large Sieve Inequality}
\newtheorem{theorem}{Theorem}[section]
\newtheorem*{theoremnonum}{Theorem}
\newtheorem{observation}[theorem]{Observation}
\newtheorem{corollary}[theorem]{Corollary}
\newtheorem{lemma}[theorem]{Lemma}
\newtheorem{prop}[theorem]{Proposition}
\newtheorem{prob}[theorem]{Problem}
\newtheorem{lemma*}{Lemma}
\newtheorem{defn}[theorem]{Definition}
\newtheorem{de*}{Definition}
\newtheorem{remark}[theorem]{Remark}
\newtheorem*{rem}{Remark}
\newtheorem*{lem}{Lemma}

\newenvironment{letter}{\begin{enumerate}[a)]}{\end{enumerate}}

\newlist{num}{enumerate}{1}
\setlist[num, 1]{label = (\alph*)}
\newcommand{\myitem}{\item}

\renewcommand{\S}{\mathcal{S}}
\newcommand{\Prob}{\mathbb{P}}
\newcommand{\E}{\mathbb{E}}
\newcommand{\Q}{\mathbb{Q}}
\newcommand{\R}{\mathbb{R}}
\newcommand{\N}{\mathbb{N}}
\newcommand{\Z}{\mathbb{Z}}
\newcommand{\F}{\mathcal{F}}

\newcommand{\C}{\mathbb{C}}
\newcommand{\Cl}{\operatorname{Cl}}
\newcommand{\disc}{\operatorname{disc}}
\newcommand{\Gal}{\operatorname{Gal}}
\newcommand{\Img}{\operatorname{Im}}
\renewcommand{\sl}{\operatorname{SL}}
\newcommand{\eps}{\varepsilon}
\newcommand{\ind}{\mathbbm{1}}
\renewcommand{\phi}{\varphi}
\newcommand{\hsum}{\sum\nolimits^h}

\newcommand\ang[1]{\left\langle#1\right\rangle}
\renewcommand{\L}{\mathbf{L}}
\newcommand{\sym}{\mathrm{Sym}}

\renewcommand{\d}{\partial}
\newcommand*\Lapl{\mathop{}\!\mathbin\bigtriangleup}
\newcommand{\mat}[4] {\left(\begin{array}{cccc}#1 & #2\\ #3 & #4 \end{array}\right)}
\newcommand{\mattwo}[2] {\left(\begin{array}{cc}#1\\ #2 \end{array}\right)}
\newcommand{\pphi}{\varphi}
\newcommand{\eqmod}[1]{\overset{\bmod #1}{\equiv}}
\renewcommand{\O}{\mathrm{O}}
\renewcommand{\subset}{\subseteq}
\newcommand{\fr}[1]{\mathfrak{#1}}
\newcommand{\p}{\mathfrak{p}}
\newcommand{\m}{\mathfrak{m}}
\newcommand{\supp}{\mathrm{supp}}
\newcommand{\limto}[1]{\xrightarrow[#1]{}}

\newcommand{\under}[2]{\mathrel{\mathop{#2}\limits_{#1}}}

\newcommand{\underwithbrace}[2]{  \makebox[0pt][l]{$\smash{\underbrace{\phantom{%
    \begin{matrix}#2\end{matrix}}}_{\text{$#1$}}}$}#2}

\title{\sc  Convergence to the Plancherel measure of Hecke Eigenvalues}
\maketitle

\begin{abstract}
We give improved uniform estimates for the rate of convergence to Plancherel measure of Hecke eigenvalues of holomorphic forms of weight $2$ and level $N$. These are applied to determine the sharp cutoff for the non-backtracking random walk on arithmetic Ramanujan graphs and to Serre's problem of bounding the multiplicities of modular forms whose coefficients lie in number fields of degree $d \geq 1$.

\end{abstract}

\section{Introduction}
It is well known that the distribution of Hecke eigenvalues of modular forms at primes $p_1, \ldots, p_r$ converges to the product of the corresponding $p$-adic Plancherel measures as one varies over certain families (\cite{sar1}, \cite{serre}, \cite{CDF}). Our aim in this paper is to establish uniform rates on this convergence and to apply these to problems of sharp cutoff for random walks on Ramanujan graphs (\cite{ns}) and to the factorization of the Jacobian of the modular curve $X_0(N)$ as in \cite{serre}.
 
The Eichler-Selberg trace formula expresses the trace of the Hecke operator $T_n$ on the space $S(N) := S_2(N)$ of holomorphic cusp forms of weight $2$ for $\Gamma_0(N)$ in terms of class numbers of binary quadratic forms. Using this, one can show (\cite{serre}, $\text{Prop. }4$) that as $f$ runs over a Hecke basis $H(N)$ of such eigenforms with eigenvalues 
$$T_n f =: \lambda_f(n) \sqrt{n} \cdot f, \ \ (n, N) = 1, $$
we will have 
\begin{align}\label{one}
\frac{1}{\abs{H(N)}} \abs{\sum_{f \in H(N)} \lambda_f(n) - \frac{\delta_{N}(n, \square)}{12}\psi(N)}\ll_{n} (n/N)^{1/2}\cdot d(N),
\end{align}
where
\begin{align*}
\delta_{N}(n, \square) := 
\begin{cases} 1 &\text{ if }n \text{ is a square modulo }N, \\
0 &\text{ otherwise,} 
\end{cases}
\end{align*} 
$d(N) := \sum_{d | N} 1 \ll_\eps N^\eps$ is the divisor function, and $\psi(N) := N\prod_{p | N}(1 + 1/p)$ is the Dedekind psi function. Murty and Sinha (\cite{MS}) give explicit and effective bounds in (\ref{one}).

To extend the range of $n$ for which the left-hand side of (\ref{one}) goes to $0$ with $N$, we introduce and remove $1/L(1, \sym^2 f)$ weights into the sum. This allows us to use the Petersson trace formula, replacing class numbers with Kloosterman sums, which enjoy sharp bounds coming from the arithmetic geometry of curves (\cite{weil}). This technique applied to a similar problem is outlined in \cite{rudnickletter} and is used on other problems in \cite{crelle} and \cite{ILS}. It allows us to double the exponent in the range of $n$, which is crucial for certain applications.

 In what follows, our aim is to establish sharp estimates, and to simplify the analysis, we assume that $N$ is prime. In much of what we do, this assumption can be removed. We address this further in the final section \ref{lastsection}. 
\begin{thm}\label{thmone}
Let $\eps > 0$ and let $m, n$ be integers coprime to $N$. Then:
$$\frac{1}{\abs{H(N)}}\sum_{f \in H(N)} \lambda_f(m) \lambda_f(n) = \sum_{\substack{d | (m, n) \\ d^2 k^2 = mn}} \frac{1}{k} + \O_\eps \left( \frac{(mn)^{1/8}}{N^{1/2}} (mnN)^{\eps} \right).$$
\end{thm}
With this theorem, we can formulate and prove a corresponding uniform convergence to the Plancherel measure. Let $r \geq 1$, and for $\ell_1, \ldots, \ell_r \geq 0$, let $\mathcal{P}^{\ell_1, \ldots, \ell_r}$ denote the set of polynomials in $x_1, \ldots, x_r$ of degrees at most $\ell_1, \ldots, \ell_r$, respectively, that is,
$$\mathcal{P}^{\ell_1, \ldots, \ell_r} := \Big\{\sum_{j_1 = 0}^{\ell_1} \cdots \sum_{j_r = 0}^{\ell_r}a_{j_1, \ldots, j_r} x_1^{j_1} \cdots x_r^{j_r}\  \large \vert \  a_{j_1, \ldots, j_r} \in \C\Big\}.$$ For $p \nmid N$, let $\theta_f(p) \in [0, \pi]$ be such that 
$$\lambda_f(p) = 2 \cos (\theta_f(p))$$ (such a $\theta_f$ exists because of self-adjointness of $T_p$ and thanks to the Ramanujan bound $\abs{\lambda_f(p)} \leq 2$ due to Eichler \cite{Eichlerramanujan}).
Let $\mu_p$ be the $p$-adic Plancherel measure:
\begin{align}
d \mu_p:=\frac{2}{\pi} \cdot \dfrac{(p + 1) \sin^2 \theta}{(p^{1/2} + p^{-1/2})^2 - 4 \cos^2 \theta } d \theta.
\end{align}
We have the following uniform convergence result: 
\begin{thm}\label{thmtwo}
Let $r \geq 1$ and $\eta > 0$. Then uniformly for $p_1^{\ell_1} \cdots p_r^{\ell_r} < N^{2 -\eta}$ and $P \in \mathcal{P}^{\ell_1, \ldots, \ell_r}$, 
$$\frac{1}{\abs{H(N)}} \sum_{f \in H(N)} \abs{P\left(\cos \theta_f(p_1), \ldots, \cos \theta_f(p_r)\right)}^2 = ( 1 + o(1)) \int_0^\pi \cdots \int_0^\pi \abs{P}^2d\mu_{p_1} \dots d \mu_{p_r}$$ as $N \to \infty$. 

\end{thm}
This result with an exponent of $N$ larger than $1$ (which corresponds to $mn$ going up to $N^{2 - \eps}$ in Theorem \ref{thmone}) is what is needed to settle the cutoff window for the non-backtracking random walks on Ramanujan graphs (\cite{ns}). In fact, it yields the conjectured asymptotics for the variance for these walks (see end of section \ref{sectionthree}).

Another application of Theorem \ref{thmtwo} is to multiplicities of $f$'s in a Hecke basis with given $\lambda_f(p)$'s for $p \in \{p_1, \ldots, p_r\}$. Let $s(N) := \abs{H(N)} = \mathrm{dim}\  S(N)$, so for $N$ prime, $s(N) = \floor{\frac{N+1}{12}} - 1$ when $N \equiv 1 \pmod{12}$ and $s(N) = \floor{\frac{N+1}{12}} $ otherwise. For $\phi \in S(N)$, let 
$$M_N(y, \phi) := \#\abs{ \left\{ f \in H(N) : \lambda_f(p) = \lambda_\phi(p) \text{ for } p \leq y, (p, N) = 1\right \}}.$$ 
For a fixed $y$, Theorem \ref{thmtwo} implies that uniformly in $\phi$, 
\begin{align}\label{logs}
M_N(y, \phi) \ll \frac{s(N)}{(\log N)^r},
\end{align}
where $r = \pi(y)$ is the number of primes up to $y$. 
 If $y$ is allowed to increase with $N$, then one can exploit that for $f$ in the set defining $M_N(y, \phi)$, we also have $\lambda_f(m) = \lambda_\phi(m)$ for all $y$-smooth numbers $m$ (which are numbers all of whose prime factors are at most $y$). This allows one to improve (\ref{logs}) vastly. 

Such an argument using the large sieve for Dirichlet characters is due to Linnik (\cite{linnik}). In the modular form setting, Duke and Kowalski (\cite{DK}) establish that the number of non-monomial newforms of square-free level up to $N$ that have prescribed eigenvalues $\lambda_\phi(p)$ at primes $p \leq y = (\log N)^\beta$ satisfies
$$M_{\leq N}(y, \phi)^{\#} \ll_\beta N^{10/\beta + \eps},$$ which is non-trivial for $\beta > 5$. Lau and Wu (\cite{LauWu}) show that for $y = C \log N$ with $C$ a large constant, there is $c > 0$ s.t.
$$M_N(y, \phi) \ll \exp \left( \dfrac{- c \log N}{\log \log N}\right) s(N).$$ Our interest is in smaller $y$'s, namely $y = (\log N)^\beta$ with $0  < \beta < 1$. 

\begin{thm}\label{thmthr}
Fix $\beta \in (0, 1)$. Then for $y = (\log N)^\beta$ and uniformly in $\phi$,
$$M_N(y, \phi) \leq \exp\left(-\frac{1 - \beta}{\beta} (\log N)^\beta + o\left((\log N)^\beta\right)\right) s(N)$$ as $N \to \infty$. 
\end{thm}

We apply this to a question of Serre (\cite{serre}). Assume that all $f =\sum_{n \geq 1} a(n) e(nz)  \in H(N)$ are normalized so $a(1) = 1$. The Fourier coefficients $a(n)$ are algebraic integers in a totally real field of degree $d(f)$. For $d \geq 1$, let $s(N)_d$ denote the number of $f$'s for which $d(f) = d$. Serre shows that for $d$ fixed, $s(N)_d = o(s(N))$ (see also \cite{MS}, Theorem $5$), and asks for stronger upper bounds. Theorem \ref{thmthr} implies such a bound.
\begin{thm}\label{thmfour}
Fix $d \geq 1$ and $\beta < \frac{2}{d + 2}$. Then as $N \to \infty$,  
$$s(N)_d \leq\exp( - c(d, \beta) (\log N)^\beta + o((\log N)^\beta)) s(N), $$ where $c(d, \beta) := (1 - \beta)/\beta - d/2> 0.$

\end{thm}

This falls short of Serre's conjecture, which asserts that Theorem \ref{thmfour} holds for $\beta = 1$ and $c = c(d) > 0$ (i.e., $s(N)_d \ll s(N)^{\alpha}$ for some $\alpha < 1$). \\
\\
{\textbf{Acknowledgments:}} We thank L. Alp\"oge, G. Harcos, E. Nestoridi, and Z. Rudnick for discussions related to the questions addressed in this paper.
\section{Weight Removal in the Petersson Formula}

Throughout this section we assume $N$ is prime. Let $H(N)$ denote a simultaneous eigenbasis of Hecke operators $T_k, (k, N) = 1$, acting on the space $S(N)$ of dimension $s(N)$ of weight $2$ level $N$ cusp forms for $\Gamma_0(N)$, and for $f \in H(N)$, let $a_f(n)$ and $\lambda_f(n)$ be such that
$$f(z)  = \sum_{n \geq 1} a_f(n) e(nz) = \sum_{n \geq 1} \sqrt{n}\  \lambda_f(n) e(nz),$$ where $e(z) := e^{2 \pi i z}$. Assume $f$ are normalized so $a_f(1) = 1$.

Our starting point for this section is the Petersson trace formula estimated via the Weil bound on Kloosterman sums, as presented in \cite{ILS}, Corollary $2.2$ or \cite{IK}, Corollary $14.24$: 
\begin{ptsn}\label{pts} 
With $H(N)$ as above, $(mn, N) = 1$, and $\eps > 0$, 
\begin{align}\label{petersson}
\sideset{}{^h}\sum_{f \in H(N)} \lambda_f(m) \lambda_f(n) = \delta(m, n) + \O_\eps\left( \frac{(mn)^{1/4}}{N} (mnN)^\eps\right).
\end{align}
Here the $h$ superscript signifies adding "harmonic" weights: $\sum_{f \in H(N)}^h \alpha_f = \frac{1}{4 \pi}\sum_{f \in H(N)} \frac{\alpha_f}{\norm{f}^2}$, where $\norm{\cdot}$ denotes the Petersson norm. 
\end{ptsn}

We derive Theorem \ref{thmone} from the Petersson formula by removing the harmonic weights. The Petersson norm is related to the special value of the symmetric square $L$-function at $1$ (\cite{ILS}, Lemma $2.5$) via:
$$4 \pi \norm{f}^2 = \frac{s(N)}{\zeta(2)} L(\sym^2 f, 1)$$ (where $L(\sym^2 f, s) = \zeta(2s)(1 - N^{-2s}) \sum_{n \geq 1} \lambda_f(n^2) n^{-s} := \sum_{n = 1}^\infty \lambda_{\mathrm{sym}^2 f}(n)n^{-s}$), so 

\begin{align}\label{form}
\frac{1}{s(N)}\sum_{f \in H(N)} \lambda_f(m) \lambda_f(n) &= \sideset{}{^h}\sum_{f \in H(N)} \lambda_f(m) \lambda_f(n) \cdot\frac{4 \pi \norm{f}^2 }{s(N)} =  \sideset{}{^h}\sum_{f \in H(N)} \lambda_f(m) \lambda_f(n) \frac{L(\sym^2 f, 1)}{\zeta(2)} ,
\end{align}
and to prove Theorem \ref{thmone}, we need to derive a suitable approximation for $L(\sym^2 f, s)$. 

Let $\Psi(x) \geq 0$ be a smooth function supported on $[-1, 1]$ with $\Psi(0) = 1$. The Mellin transform 
$$\tilde{\Psi}(s) = \int_0^\infty \Psi(x) x^s \frac{dx}{x}$$ is an analytic function of $s = \sigma + i t$ for $\sigma> -1$, except for a simple pole at $0$ with residue $1$, and decreases rapidly as $\abs{t} \to \infty$ for $ -1 \leq \sigma \leq 2$.

For a parameter $x>0$, let 
\begin{align*}
\textbf{A}:&= \frac{1}{2 \pi i} \int_{\sigma = 2} L(\sym^2 f, s + 1) x^s \tilde{\Psi}(s) ds \\
&=   \frac{1}{2 \pi i} \int_{\sigma = 2}\sum_{\nu \geq 1}\frac{\lambda_{\mathrm{sym}^2 f}(\nu)}{\nu^{s+1}}x^s \tilde{\Psi}(s) ds  \\
&=  \sum_{\nu \geq 1} \frac{\lambda_{\mathrm{sym}^2 f}(\nu)}{\nu } \cdot \frac{1}{2 \pi i}\int_{\sigma = 2}\left(\frac{\nu}{x}\right)^{-s} \tilde{\Psi}(s) ds  \\
&= \sum_{\nu \leq x} \frac{\lambda_{\mathrm{sym}^2 f} (\nu)}{\nu} \Psi\left( \frac{\nu}{x}\right)
\end{align*}
by the Mellin inversion theorem.
Shifting the integral defining \textbf{A} to $\Re(s) = -1/2$ picks up the simple pole of $\tilde{\Psi}$ at $s = 0$, so by the residue theorem, 

\begin{align}\label{lfunction}
L(\sym^2 f, 1) = \sum_{\nu \leq x} \frac{\lambda_{\mathrm{sym}^2 f} (\nu)}{\nu} \Psi\left( \frac{\nu}{x}\right) + R(f, x), 
\end{align} where 
$$R(f, x) := - \frac{1}{2 \pi} \int_{-\infty}^\infty L(\sym^2 f, 1/2 + i t) x^{-\frac{1}{2} + i t} \tilde{\Psi}(-1/2 + it) dt.$$
Now, by Cauchy-Schwarz,
\begin{align*}
\abs{R(f, x)}^2  &= \frac{1}{4 \pi^2  x} \abs{\int_{- \infty}^\infty L(\sym^2 f, 1/2 + it) x^{it} \tilde{\Psi}(-1/2 + it) dt}^2 \\
&\ll x^{-1} \int_{-\infty}^\infty \abs{L(\sym^2 f, 1/2 + it)}^2 \cdot \big|\tilde{\Psi}(-1/2 + it)\big| dt,
\end{align*}
so 
$$ \sideset{}{^h}\sum_{f \in H(N)} \abs{R(f, x)}^2 \ll x^{-1} \int_{-\infty}^\infty \big| \tilde{\Psi}(-1/2 + it) \big|  \sideset{}{^h}\sum_{f \in H(N)}\abs{L(1/2 + it, \sym^2 f)}^2  dt.$$
According to the Lindel\"{o}f on average result due to Iwaniec and Michel for this family of $L$-functions (\cite{IM}), 
$$  \sideset{}{^h}\sum_{f \in H(N)}\abs{L(1/2 + it, \sym^2 f)}^2  \ll_\eps N^{\eps} (\abs{t} + 1)^8, $$ so 
\begin{align}\label{remainder}
 \sideset{}{^h}\sum_{f \in H(N)} \abs{R(f, x)}^2 \ll_\eps x^{-1} N^{ \eps}.
\end{align}
Substituting (\ref{remainder}) and (\ref{lfunction}) into (\ref{form}),
\begin{align}
\frac{1}{s(N)} \sum_{f \in H(N)} \lambda_f(m) \lambda_f(n) &=  \frac{1}{\zeta(2)} \sideset{}{^h}\sum_{f \in H(N)} \lambda_f(m) \lambda_f(n)\left( \sum_{\nu \leq x} \frac{\lambda_{\mathrm{sym}^2 f} (\nu)}{\nu} \Psi\left( \frac{\nu}{x}\right) + R(f, x)\right)   \notag \\
&=\frac{1}{\zeta(2)}( \textbf{I} + \textbf{II}) , \label{expression}
\end{align} 
where 
$$\textbf{I} = \sum_{\nu \leq x} \frac{\lambda_{\mathrm{sym}^2 f} (\nu)}{\nu} \Psi\left( \frac{\nu}{x}\right)  \sideset{}{^h}\sum_{f \in H(N)}\lambda_f(m) \lambda_f(n)$$
and 
$$\abs{\textbf{II}} \leq \sideset{}{^h}\sum_{f \in H(N)} \abs{\lambda_f(m)\lambda_f(n)}\abs{R(f, x)} \ll_\eps {N^\eps}\left( \sideset{}{^h}\sum_{ \ f \in H(N)} (mn)^\eps\right)^{1/2}\left(  \sideset{}{^h}\sum_{ \ f \in H(N)} \abs{R(f, x)}^2\right)^{1/2} \ll_\eps \frac{(m n N)^\eps}{\sqrt{x}},$$
where we used Cauchy-Schwartz and (\ref{remainder}).

To estimate \textbf{I}, we use Hecke relations 
$$\lambda_f(m) \lambda_f(n) = \sum_{d | (m, n)} \lambda_f\left(\frac{mn}{d^2} \right)$$ and the formula 
$$\lambda_{\sym^2 f}(\nu) = \sum_{\substack{t^2 k = \nu \\ (t, N) = 1}} \lambda_f(k^2).$$
From this,
$$\textbf{I} = \sum_{t^2 k \leq x} \frac{\Psi\left( \frac{t^2k}{x}\right)}{t^2 k} \sum_{d | (m, n)} \sum_{f \in H(N)}^h \lambda_f\left( \frac{mn}{d^2}\right) \lambda_f \left( k^2\right),$$ so by the Petersson formula, 
\begin{align*}
\textbf{I} &= \sum_{t^2 k \leq x}\frac{\Psi\left( \frac{t^2k}{x}\right)}{t^2 k}{t^2 k } \sum_{\substack{d | (m, n) \\ mn = d^2k^2}} 1 + \O_\eps\left( \sum_{t^2 k \leq x} \frac{(mn)^\eps}{t^2 k}\frac{(mnk^2)^{1/4}}{N}(Nmn)^\eps\right)  \\
&= \left(\sum_{\substack{d | (m, n) \\ mn = d^2k^2}} \frac{1}{k} \right)  \sum_{t^2 k \leq x/k}\frac{\Psi\left( \frac{t^2k}{x}\right)}{t^2  } + \O_\eps \left( \frac{x^{1/2}(mn)^{1/4}}{N}(mnN)^\eps\right) \\
&= \zeta(2) \left(\sum_{\substack{d | (m, n) \\ mn = d^2k^2}} \frac{1}{k}\right) + \O_\eps \left( \frac{x^{1/2}(mn)^{1/4}}{N}(mnN)^\eps\right).
\end{align*}
Combining estimates of \textbf{I} and \textbf{II} with (\ref{expression}), 

$$\frac{1}{s(N)} \sum_{f \in H(N)} \lambda_f(m) \lambda_f(n)  = \sum_{\substack{d | (m, n) \\ mn = d^2k^2}} \frac{1}{k} + \O_\eps((mn N)^\eps/x^{1/2}) + \O_\eps \left( \frac{x^{1/2}(mn)^{1/4}}{N}(mnN)^\eps\right). $$
Choosing $x = \frac{N}{(mn)^{1/4}}$, we finish the proof of Theorem \ref{thmone}.

\section{Convergence to the Plancherel Product Measure}\label{sectionthree}
In this section we address Theorem \ref{thmtwo}. Fix an integer $r > 0$, let $\ell_1, \ldots, \ell_r > 0$, and let $p_1, \ldots, p_r$ be distinct primes with $(p_j, N) = 1$. 

Consider a polynomial $P(x_1, \ldots, x_r) \in \C[x_1, \ldots, x_r]$ of degree at most $\ell_i$ in $x_i$ for $1 \leq i \leq r$.
For $n \geq 0$,  let 
$$U_n(\cos \theta) := \frac{\sin((n + 1) \theta)}{\sin \theta}$$ be the $n$th Chebyshev polynomial of the second kind. $U_n$ is a degree $n$ polynomial in $\cos \theta$ with real coefficients, so we can find $a_{t_1, \ldots, t_r} \in \C$ such that 

\begin{align}\label{polycoefs}
P(x_1, \ldots, x_r) = \sum_{t_1 = 0}^{\ell_1} \cdots \sum_{t_r = 0}^{\ell_r} a_{t_1, \ldots, t_r} U_{t_1}(x_1) \cdots U_{t_r}(x_r).
\end{align}
Moreover, Hecke relations imply that for $(p, N) = 1$, $U_n(\lambda(p)) = \lambda(p^n)$. 
From this,
\begin{align*}
 \abs{P(\theta_f(p_1), \ldots, \theta_f(p_r))}^2   & = \sum_{t_1, s_1 = 0}^{\ell_1} \cdots \sum_{t_r, s_r = 0}^{\ell_r} a_{t_1, \ldots, t_r} \overline{a_{s_1, \ldots, s_r} } U_{t_1}(\theta_f(p_1)) \overline{U_{s_1}(\theta_f(p_1))} \cdots  U_{t_r}(\theta_f(p_r)) \overline{U_{s_r}(\theta_f(p_r))}= \\
 &  =\sum_{t_i, s_i}a_{t_1, \ldots, t_r} \overline{a_{s_1, \ldots, s_r} } \lambda_f(p_1^{t_1}) \overline{\lambda_f(p_1^{s_1})} \cdots  \lambda_f(p_r^{t_r}) \overline{\lambda_f(p_r^{s_r})} \\
 &  =\sum_{t_i, s_i}a_{t_1, \ldots, t_r} \overline{a_{s_1, \ldots, s_r} } \lambda_f(p_1^{t_1} \cdots p_r^{t_r}) \overline{\lambda_f(p_1^{s_1} \cdots p_r^{s_r})},
\end{align*}
so by Theorem \ref{thmone}, 

\begin{align}\label{oneplustwo}
\frac{1}{\abs{H(N)}} \sum_{f \in H(N)}  \abs{P(\theta_f(p_1), \ldots, \theta_f(p_r))}^2   &=\sum_{t_i, s_i} a_{t_1, \ldots, t_r} \overline{a_{s_1, \ldots, s_r} }\cdot  \frac{1}{\abs{H(N)}} \sum_{f \in H(N)}\lambda_f(p_1^{t_1} \cdots p_r^{t_r}) \overline{\lambda_f(p_1^{s_1} \cdots p_r^{s_r})} \notag \\ 
&= \textbf{I} + \textbf{II},
\end{align}
where 
$$\textbf{I} = \sum_{\substack{ s_i, t_i \leq \ell_i \\ m = p_1^{t_1}\cdots p_r^{t_r}\\ n = p_1^{t_1}\cdots p_r^{t_r}}} a_{t_1, \ldots, t_r} \overline{a_{s_1, \ldots, s_r}}\sum_{\substack{d |(m ,n)\\ mn = d^2 k^2}}\frac{1}{k},$$ and 

\begin{align*}
\textbf{II} &\ll_\eps \frac{1}{\sqrt{N}}\sum_{\substack{ s_i, t_i \leq \ell_i \\ m = p_1^{t_1}\cdots p_r^{t_r}\\ n = p_1^{t_1}\cdots p_r^{t_r}}} \abs{a_{t_1, \ldots, t_r} \overline{a_{s_1, \ldots, s_r}}}(mn)^{1/8} (mnN)^\eps\\
&\ll \frac{1}{\sqrt{N}}\left( \sum_{0 \leq t_i \leq \ell_i} \abs{a_{t_1, \ldots, t_r}}p_1^{t_1/8} \cdots p_r^{t_r/8} \right)^2 (p_1^{2 \ell_1} \cdots p_r^{2 \ell_r} N)^\eps = (*).
\end{align*}
Applying Cauchy-Schwartz and summing the geometric series, 
$$(*)  \leq  \frac{(p_1^{2 \ell_1} \cdots p_r^{2 \ell_r} N)^\eps }{\sqrt{N}} \hspace{-0.05in} \sum_{0 \leq t_i \leq \ell_i} p_1^{t_1/4} \cdots p_r^{t_r/4}  \hspace{-0.05in} \sum_{0 \leq t_i \leq \ell_i} \abs{a_{t_1, \ldots, t_r}}^2 \leq  (p_1^{2 \ell_1} \cdots p_r^{2 \ell_r} N)^\eps \left(\frac{ p_1^{\ell_1} \cdots p_r^{\ell_r}}{N^2}\right)^\frac{1}{4}\hspace{-0.1in} \sum_{0 \leq t_i \leq \ell_i} \abs{a_{t_1, \ldots, t_r}}^2.$$
Let $\mu_\infty(\theta):= \frac{2}{\pi} \sin^2 \theta d \theta$ be the Sato-Tate measure on $[0, \pi]$. From the definition of $d \mu_p$, it follows that
$$d \mu_\infty \cdots d \mu_\infty \ll_r d \mu_{p_1} \cdots d \mu_{p_r}.$$ Hence, from (\ref{polycoefs}) and the orthonormality of $U_n$ with respect to $d \mu_\infty$, we have 

$$ \sum_{0 \leq t_i \leq \ell_i} \abs{a_{t_1, \ldots, t_r}}^2 \ll \norm{P}^2_{\mu_{p_1}, \ldots, \mu_{p_r}}.$$ We conclude that  
$$\textbf{II} \ll_\eps  (p_1^{ \ell_1} \cdots p_r^{ \ell_r} N)^{2\eps} \left(\frac{ p_1^{\ell_1} \cdots p_r^{\ell_r}}{N^2}\right)^\frac{1}{4}\norm{P}^2_{\mu_{p_1}, \ldots, \mu_{p_r}}.$$
It remains to evaluate \textbf{I}. To interpret \textbf{I} as the integral against the Plancherel measure, we need the following observation:
\begin{prop}\label{innerproduct}
Let $m, n \geq 0$, and let $d \mu_p$ be the $p$-adic Plancherel measure. Then:  
\begin{align}\label{innerpr}
\int_0^\pi U_n(\theta) U_m(\theta) d \mu_p = 
\left[
\begin{array}{ll}
      \dfrac{p}{(p-1)} \left(\dfrac{1}{p^{\frac{\abs{m - n}}{2}}} - \dfrac{1}{p^{\frac{m + n}{2} + 1}} \right) &\text{if } m \equiv n \pmod{2} \\
    0 & \text{otherwise.}
\end{array}
\right.
\end{align}
\end{prop}
We leave the proof until the end of the section. From (\ref{polycoefs}), 
\begin{align}\label{prodmeasureinner}
\int_{[0, \pi]^n} \abs{P(\theta_1, \ldots, \theta_r)}^2 d \mu_{p_1} \cdots d \mu_{p_r} &= \sum_{t_i, s_i \leq \ell_i} a_{t_1} \overline{a_{s_1}} \cdots a_{t_r} \overline{a_{s_r}}\int_{[0, \pi]^n} \hspace{-0.1in} U_{t_1}(\theta_1) U_{s_1}(\theta_1)\cdots U_{t_r}(\theta_r) U_{s_r}(\theta_r) d \mu_{p_1} \cdots  d \mu_{p_r} \notag \\
&= \sum_{t_i, s_i \leq \ell_i} a_{t_1} \overline{a_{s_1}} \cdots a_{t_r} \overline{a_{s_r}} \cdot \prod_i \int_0^\pi U_{t_i}(\theta) U_{s_i}(\theta) d \mu_{p_i}.
\end{align}
We substitute the inner product (\ref{innerpr}) into (\ref{prodmeasureinner}). Observe that 
\begin{align*}
 \frac{p}{(p-1)} \left(\frac{1}{p^{\frac{\abs{m - n}}{2}}} - \frac{1}{p^{\frac{m + n}{2} + 1}} \right)  &= \frac{1}{p^{\frac{\abs{m - n}}{2}}} + \frac{1}{p^{\frac{\abs{m - n}}{2}} + 1} + \cdots + \frac{1}{p^{\frac{m + n}{2}}} \\
& =  \sum_{p^{\alpha}  | (p^m, p^n)}\frac{1}{p^{\frac{m + n}{2} - \alpha}} \\
& =  \sum_{\substack{ d  | (p^m, p^n) \\ d^2 k^2 = p^m p^n }}\frac{1}{k},
\end{align*}
so (\ref{innerpr}) implies

\begin{align*}
\norm{P}^2_{\mu_{p_1}, \ldots, \mu_{p_r}}= (\ref{prodmeasureinner}) &= \sum_{\substack{s_i, t_i \leq \ell_i \\ s_i \overset{\mathrm{mod} 2}{\equiv} t_i}} a_{t_1} \overline{a_{s_1}} \cdots a_{t_r} \overline{a_{s_r}} \cdot \prod_i \sum_{\substack{d  | (p_i^{t_i}, p_i^{s_i}) \\ d^2 k^2 = p_i^{t_i + s_i}}}\frac{1}{k} = \\
& \sum_{\substack{s_i, t_i \leq \ell_i \\ s_i \overset{\mathrm{mod} 2}{\equiv} t_i }} a_{t_1} \overline{a_{s_1}} \cdots a_{t_r} \overline{a_{s_r}} \cdot  \sum_{\substack{ d  | (p_1^{t_1} \cdots p_r^{t_r}, p_1^{s_1} \cdots p_r^{s_r} ) \\ d^2 k^2 = p_1^{t_1 + s_1} \cdots p_r^{t_r  + s_r}}}\frac{1}{k} = \textbf{I}.
\end{align*}
Combining \textbf{I} and \textbf{II}, 

\begin{align}\label{abc}
\frac{1}{\abs{H(N)}} \sum_{f \in H(N)}  \abs{P(\theta_f(p_1), \ldots, \theta_f(p_r))}^2  &= \norm{P}^2_{\mu_1, \ldots, \mu_r} \left( 1 + \O\left((p_1^{ \ell_1} \cdots p_r^{ \ell_r} N)^{2\eps} \left(\frac{ p_1^{\ell_1} \cdots p_r^{\ell_r}}{N^2}\right)^\frac{1}{4}\right)\right). \notag \\
\end{align}
Suppose now that the conditions of Theorem \ref{thmtwo} are met, i.e., 
$$p_1^{\ell_1} \cdots p_r^{\ell_r} < N^{2 - \eta}$$ for some $\eta > 0$. Then (\ref{abc}) is equal to
$$ \norm{P}^2_{\mu_1, \ldots, \mu_r} ( 1 + \O(N^{6 \eps} \cdot N^{-\eta/4}))= \norm{P}^2_{\mu_1, \ldots, \mu_r} ( 1 + o(1))$$ for small enough $\eps$. This concludes the proof of Theorem \ref{thmtwo}.

\begin{proof}[Proof of Proposition \ref{innerproduct}]

From the trigonometric identity for the product of sines,
\begin{align*}
\int_0^\pi U_n(\theta) U_m(\theta) d \mu_p &= \frac{ p + 1}{ 2 \pi } \int_0^\pi \frac{\sin((n + 1) \theta) \sin((m + 1)\theta)}{\left(\frac{p^{1/2}+p^{-1/2}}{2}\right)^2 -  \cos^2 \theta}d\theta \\
&= \frac{p + 1}{4\pi} \int_0^\pi \frac{\cos((m - n)\theta) - \cos ((m + n + 2)\theta)}{\frac{(p-1)^2}{4p} +  \sin^2 \theta}d\theta\\
&= \mathcal{I}(\abs{m - n}) - \mathcal{I}(m + n + 2),
\end{align*}
where $\mathcal{I}(k) := \frac{p + 1}{4 \pi} \int_0^\pi  \frac{\cos(k \theta)}{\frac{(p-1)^2}{4p} +  \sin^2 \theta}d\theta$.

Since $\sin^2(\theta) = \sin^2(\pi - \theta)$ and $\cos(k(\pi - x)) = - \cos kx$ for odd $k$, $\mathcal{I}(k) = 0$ for odd $k$, so the integral is $0$ when $m$ and $n$ have different parity. To prove the proposition, it suffices to show that for all integers $T \geq 0$, 
\begin{align}\label{induction}
\mathcal{I}(2T) = \frac{(p + 1)}{4\pi} \int_0^\pi \frac{\cos (2T\theta)}{\frac{(p - 1)^2}{4p} + \sin^2 \theta} d\theta= \frac{p}{(p-1) p^T}.
\end{align}
We prove this statement by induction.  Let $\zeta := e^{ix}$, $c := \frac{(p - 1)^2}{4p}$, and let $\alpha := 2 + 4c = p + 1/p$.  Then: 
\begin{align*}
\int_0^\pi \frac{\cos (2T\theta)}{\frac{(p - 1)^2}{4p} + \sin^2 \theta} d\theta  &= \frac{1}{2} \int_0^{2\pi} \frac{(\zeta^{2T} + \zeta^{-2 T})/2}{c+ ((\zeta - \zeta^{-1})/2i)^2 } dx  = \\
&=-\int_0^{2\pi}\frac{(\zeta^{2T} + \zeta^{-2 T})}{-4c+ (\zeta - \zeta^{-1})^2 } dx  \\
&=-\int_{S^1}\frac{(\zeta^{4T} + 1)}{\zeta^{2T-1}(\zeta^4 - \alpha \zeta^2 + 1)} d\zeta \\
&=-\int_{S^1}\frac{(\zeta^{4T} + 1)}{\zeta^{2T-1}(\zeta^2 - p)(\zeta^2 - 1/p)} d\zeta.
\end{align*}

We evaluate the integral using the residue theorem. For $T = 0$, the poles are at $\pm \sqrt{1/p}$, and both residues are equal to $\frac{1/\sqrt{p}}{(p - 1/p) \cdot 2/\sqrt{p}} =\frac{1}{p - 1/p}$, so $$\mathcal{I}(0) = (p + 1) \cdot 2 \pi(2p/(p^2-1) )/4 \pi = p/(p-1).$$ For $T = 1$, the pole at $0$ has residue $-1$ and the poles at $\pm 1/\sqrt{p}$ have residues $\frac{1/p^2 + 1}{(1/\sqrt{p})(p - 1/p) (2/\sqrt{p})} = \frac{p^2 + 1}{2(p^2 - 1)}$, so $$\mathcal{I}(2) = (p + 1) \cdot 2\pi(-1 + (p^2 + 1)/(p^2 - 1))/4 \pi = 1/(p-1).$$
Assume now $T \geq 2$. The rational function $- \dfrac{x^{4T} + 1}{x^{2T - 1}(x^2 - p)(x^2 - 1/p)}$ has three poles inside the unit circle: $0, \omega = 1/\sqrt{p}$ and $-\omega$, and the two latter ones have the same residue. Let $A(T)$, $B(T)$ be the residues at $0$ and $\omega$ respectively. Then:
$$B(T) = - \dfrac{(\omega^2)^{2T} + 1}{2(\omega^2)^T(\omega^2 - p)} = \frac{1}{2(p - 1/p)}(p^T + 1/p^T),$$
and $A(T)$ is the coefficient of $x^{2T - 2}$ in $1/(1 - \alpha x^2 + x^4) = \sum_{r \geq 0} (x^4 - \alpha x^2)^r$. 

Notice that both $A(T)$ and $B(T)$ satisfy the recurrence relation 
$$F( T + 2) - \alpha F(T + 1) + F(T) = 0.$$ It remains to notice $\dfrac{p}{(p-1) p^T}$ satisfies the same recurrence relation, which proves (\ref{induction}).
\end{proof}

To end this section, we apply Theorem \ref{thmtwo} to the question of sharp cutoff of random walks on certain Ramanujan graphs. Let $\ell$ be a fixed prime and $p$ a large prime, $p \equiv 1 \pmod{12}$ (the notation here is made to conform with \cite{cgl}). The Brandt-Ihara-Pizer "super singular isogeny graphs," $G(p, \ell)$, are $d := \ell + 1$ regular graphs on $n := \frac{p-1}{12} + 1$ vertices (see \cite{cgl}, page $4$, for a description). The non-trivial eigenvalues of $G(p, \ell)$ are the numbers $2\sqrt{\ell}\cos(\theta_f(\ell)) $ for $f \in H(N)$ ($N = p$ in our notation). The $G(p, \ell)$'s are $d$-regular Ramanujan graphs on $n$ vertices. The $L^2$-variance, $W_2(t)$, for the $t$-step non-backtracking random walk on $G(p, \ell)$ is given by (see \cite{ns}, page $13$)
$$W_2(t) = \frac{\ell^t}{n} \sum_{f \in H(N)} \abs{R_t(\cos \theta_f(\ell))}^2,$$ where $R_t$ is the $t^\text{th}$ orthogonal polynomial on $[0, \pi]$ with respect to $d \mu_\ell$, normalized so that 
$$\int_0^\pi \abs{R_t(\theta)}^2 d \mu_\ell(\theta) = \frac{\ell + 1}{\ell}.$$ Applying Theorem \ref{thmtwo} with $r = 1$, $p_1 = \ell$, and $\ell_1 = t$ yields that uniformly for $t < (2 - \eta) \log_\ell n$, 
$$W_2(t) \sim ( \ell + 1) \ell^{t - 1} \text{ as } n \to \infty.$$ Note that $N(t)$, the number of non-backtracking walks of length $t$, is $(\ell + 1) \ell^{t-1}$, so that 
$$W_2(t) \sim N(t) \text{ for } t < ( 2 - \eta) \log_\ell n.$$ This proves conjecture $1.8$ in \cite{ns} for graphs $G(p, \ell)$. For the application to bounded window cutoff one needs $t$ to be as large as $(1 + \eps) \log_\ell n$, which is provided by the key doubling of the degree of $P$ in Theorem \ref{thmtwo}. In order to prove Conjecture $1.8$ in \cite{ns} for the more general Ramanujan graphs constructed using modular forms, one would need to identify the images of division algebras in $H(N)$ under the Jacquet-Langlands correspondence and restrict the sums in Theorem $2$ to those forms.

\section{\bfseries{{Multiplicity of Eigenvalue Tuples}}}\label{multipli}
Recall that for a fixed prime level $N$ and $\phi \in S(N)$ a weight $2$ holomorphic cusp form for $\Gamma_0(N)$, we let $M_N(y, \phi)$ be the multiplicity of the tuple of eigenvalues of $\phi$ at primes up to $y$ in a Hecke basis $H(N)$, i.e. 
$$M_N(y, \phi) := \#\abs{ \left\{ f\in H(N) : \lambda_f(p) = \lambda_\phi(p) \text{ for } p \leq y, (p, N) = 1\right \}}.$$ In this section we bound $M_N(y, \phi)$ uniformly in $\phi$ in the range $y = (\log N)^\beta$ for a fixed $\beta \in (0, 1)$. Specifically, we prove Theorem \ref{thmthr} via the large sieve and smooth number estimates.

From now on, we assume $y = o(\log N)$. We let $p_1, \ldots, p_r$ denote the first $r$ prime numbers, where $r = \pi(y)$ is the number of primes up to $y$. 

An integer $m$ is called \textit{$y$-smooth} if all primes $p | m$ satisfy $p \leq y$. The set of $y$-smooth numbers is denoted with $\mathcal{S}_y$, and the de Bruijn function $\Psi(y, M)$ is the counting function for $y$-smooth numbers up to $M$: 
$$\Psi(y, M): = \# \abs{\{m \in \mathcal{S}_y, m \leq M\}}.$$ We use the large sieve inequality as in \cite{IK}, Theorem $7.26$ (the inequality is stated in \cite{IK} for weight $k > 2$ but holds for $k = 2$ as well -- see comment after the proof): 

\begin{sv}
Let $\mathcal{F}$ be an orthonormal basis of $S(N)$, $f(z) := \sum \rho_f(n) e(nz)$ for $f \in \mathcal{F}$. Then for any complex numbers $\{c_n\}$ we have
\begin{align}\label{sieve}
\sum_{f \in \mathcal{F}} \abs{\sum_{n \leq M} \frac{c_n \rho_f(n)}{\sqrt{n}}}^2  \ll (1 + M/N) \norm{c}^2,
\end{align} where $ \norm{c}^2 = \sum_{n \leq M } \abs{c_n}^2$ and the implied constant is absolute. 
\end{sv}
We apply this with $M= N$ and 
\begin{align*}
c_n :=
\begin{cases} \overline{\lambda_\phi(n)} &\text{ if } n \in \mathcal{S}_y \\
0 &\text{ otherwise.}
\end{cases}
\end{align*}
Using that $\rho_f(n) = \sqrt{n} \cdot \lambda_f(n)  \rho_f(1)$ and the definition of $\S_y$, we have
$$M_N(y, \phi) \abs{\rho^2_\phi(1)} \Big|\sum_{\substack{ n \leq  N\\ n \in \S_y}}\abs{\lambda_\phi(n)}^2\Big|^2 \hspace{-0.05in}   \leq  \hspace{-0.3in} \sum_{\substack{f \in \mathcal{F}:\\ \lambda_f(m) =\lambda_\phi(m) \\ \text{for }m \in \S_y, m \leq N}} \hspace{-0.25in} \Big|\rho_\phi(1) \sum_{\substack{ n \leq  N\\ n \in \S_y}}\abs{\lambda_\phi(n)}^2\Big|^2 \hspace{-0.05in}    \leq   \sum_{f \in \mathcal{F}} \Big|\sum_{\substack{n \leq N\\ n \in \S_y } }\rho_f(1) \lambda_f(n) \overline{\lambda_\phi(n)}\Big|^2  \overset{\text{(\ref{sieve}) }}{\ll} \sum_{\substack{n \leq N\\ n \in \S_y}} \abs{\lambda_\phi(n)}^2,$$ which, combined with the Hoffstein-Lockhart estimate $\abs{\rho^2_\phi(1)} \gg N^{-1 }(\log N)^{-2}$ (\cite{HL}) yields

\begin{align}\label{sizeofm}
M_N(y, \phi)  \ll N(\log N)^2/ \sum_{\substack{n \leq N\\ n \in \S_y}} \abs{\lambda_\phi(n)}^2.
\end{align}
To prove Theorem \ref{thmthr}, we bound $\sum_{\substack{n \leq N\\ n \in \S_y}} \abs{\lambda_\phi(n)}^2$ away from $0$ uniformly in $\phi$. We use the following fact:

\begin{prop}
Let $k, x \in \R$. Then: 
\begin{align}\label{sins}
\max \left\{\abs{\sin kx/\sin x}, \abs{\sin ((k+1)x)/\sin x} \right\} \geq  1/2.
\end{align}
(where the functions are extended continuously to the $x$ with $\sin x  = 0$).
\end{prop}
\begin{proof}
If $\sin x = 0$, (\ref{sins}) is true since $\max\{k, k + 1\} \geq 1/2$, so assume $\sin x \neq 0$. Let $\frac{\sin^2kx}{\sin^2 x} = \eps^2$, where $0 \leq \eps < 1$ (since if $\eps \geq 1$, (\ref{sins}) clearly holds). Then:  
$$\abs{\cos kx} = \sqrt{1 - \eps^2 \sin^2x } = \sqrt{1 - \eps^2 + \eps^2 \cos^2 x},$$ so by the trigonometric identity for sine of a sum,
\begin{align*}
\abs{\frac{\sin (kx + x)}{\sin x}} &= \abs{\frac{\sin kx }{\sin x}\cos x + \cos kx} \geq  \abs{\cos kx} -\eps \abs{\cos x } =\sqrt{1 - \eps^2 + \eps^2 \abs{\cos x}^2} - \eps\abs{\cos x}.
\end{align*} 
It remains to minimize $f_\eps(t): = \sqrt{1 - \eps^2 + \eps^2 t^2} - \eps t$ for $t \in [0, 1]$. The function $f_\eps(t)$ has a non-vanishing derivative in this interval when $\eps < 1$, so 
$$\abs{\frac{\sin ((k+1)x)}{\sin x}} \geq \min\{f_\eps(0), f_\eps(1) \} = \min \{\sqrt{1 - \eps^2}, 1 - \eps \} = 1 - \eps,$$ and so 
$$\max \left\{\abs{\frac{\sin kx}{\sin x}}, \abs{\frac{\sin ((k+1)x)}{\sin x}} \right\}  \geq \max\{\eps, 1 - \eps\} \geq 1/2.$$
\end{proof}
Using the Hecke relation $$\lambda_\phi(p^k) = \frac{\sin((k + 1)\theta_\phi(p))}{\sin( \theta_\phi(p))},$$ (\ref{sins}) implies that 
$$\max\{\abs{\lambda_\phi(p^k)}, \abs{\lambda_\phi(p^{k + 1})}\} \geq 1/2$$ for all $\phi$ and $k \geq 0$. Since $\lambda_\phi(n)$ is multiplicative, for any $r$-tuple $(\alpha_1, \ldots, \alpha_r)$ of non-negative integers, there are $(\delta_1, \ldots, \delta_r)\subset \{0, 1\}^r$ such that 
\begin{align}\label{fourminusr}
\abs{\lambda_\phi(p_1^{2\alpha_1 + \delta_1} \cdots p_r^{2 \alpha_r + \delta_r})}^2\geq 4^{-r}.
\end{align} The set $\S_y$ of all $y$-smooth numbers is a disjoint union of sets
$$\mathcal{E}_{\alpha_1, \ldots, \alpha_r} := \{p_1^{2\alpha_1 + \delta_1} \cdots p_r^{2 \alpha_r + \delta_r}| \delta_i \in \{0, 1\}\}$$ of size $2^r$,  and (\ref{fourminusr}) implies that each $\mathcal{E}_{\alpha_1, \ldots, \alpha_r}$ contains an element $t$ with $\abs{\lambda_\phi(t)}^2 \geq 4^{-r}$. Moreover, for every $s \in S_y \cap [0, N/(p_1 \cdots p_r)]$, the set $\mathcal{E}_\alpha$ containing $s$ is fully contained in $\S_y \cap [0, N]$. Hence, there are at least $\Psi[y, N/(p_1 \cdots p_r)]/2^r$ sets $\mathcal{E}_\alpha$ fully contained in $\S_y \cap [0, N]$, so
\begin{align}\label{totoo}
\sum_{\substack{n \leq N\\ n \in \S_y}} \abs{\lambda_\phi(n)}^2 \gg 8^{-r} \Psi\left(y, \frac{N}{\prod_{p < y} p}\right) = 8^{-r} \Psi\left(y, \frac{N}{N^{\frac{( 1 + o(1))y}{\log N}}}\right),
\end{align}
where we used the prime number theorem in the second step. We use a result of Hildebrand and Tenenbaum on the size of $\Psi(y, X)$ in the range $y = o(\log X)$:
\begin{est}[{\cite{hild}, Corollary 1}]\label{hil}
 \ \\
Let $y = o(\log X)$ such that $y \to \infty$ as $X \to \infty$. Let $$\alpha(X, y):= (1 + o(1))\frac{y}{\log X \log y},$$ and $$\zeta(\alpha, y):= \prod_{p \leq y}(1 - p^{-\alpha})^{-1}.$$ Then: 
\begin{align}\label{h}
\Psi(X, y) = (1 + o(1)) X^\alpha \zeta(\alpha, y)\sqrt{\log y/(2 \pi y)}.
\end{align}
\end{est}
From (\ref{h}), 
\begin{align}\label{psi}
\log \Psi(X, y) = \alpha \log X + \log \zeta(\alpha, y) + \O(\log y) = \log \zeta(\alpha, y) + \O(y/\log y),
\end{align}
and using the Taylor series expansion,
\begin{align}
\log \zeta(\alpha, y) &= - \sum_{p \leq y} \log(1 - e^{-\alpha \log p}) \notag \\
&= -\sum_{p \leq y} \log(\alpha\log p(1 + \O(\alpha \log p))) \notag \\
&= - \pi(y) \log \alpha - \sum_{p \leq y} \log \log p + O(\alpha) \sum_{p \leq y} \log p .\label{tag}
\end{align}
Using partial summation, 
$$\abs{\sum_{p \leq y} \log \log p - \pi(Y) \log \log Y} \ll \int_2^Y\frac{dt}{\log^2 t } = \left(\mathrm{li}(t)  - \frac{t}{\log t}\right)\Big|_2^Y \ll \frac{Y}{\log^2 Y},$$ and by the prime number theorem, $\sum_{p \leq y} \log p = (1 + o(1)) y$. Hence,  
\begin{align*}
(\ref{tag}) &= \pi(y)\left( -\log y + \log \log X + \log \log y + \O(1) - \log \log y\right) + \O( y/\log^2 y) + \O(y^2/(\log y \log X))\\
&= (1 + o(1)) \frac{y}{\log y} \log \left(\frac{\log X}{y}\right), 
\end{align*}
and so (\ref{psi}) gives
\begin{align}\label{finalpsi}
\log \Psi(X, y) = (1 + o(1)) \frac{y}{\log y} \log \left(\frac{\log X}{y}\right).
\end{align}
Finally, we combine the results above. Let $X = N^{1 - (1 + o(1))\frac{y}{\log N}}$, so 
$$\log X = \log N\left(1 - (1 + o(1)) \frac{y}{\log N}\right) = \log N (1 + o(1))$$ when $y = o(\log N)$. Clearly, for such $y$ we also have $y = o(\log X)$, so it follows from Theorem \ref{hil} that (\ref{finalpsi}) holds for such $y$ and $X$. Finally, combining (\ref{finalpsi}) with (\ref{sizeofm}) and (\ref{totoo}) and using that $r = \pi(y) = o(y)$ (i.e. $\log( 8^{r}) = o(y)$), we see that

$$\log (M_N(y, \phi)/N) \leq \O(\log \log N) + o(y) + (1 + o(1))\frac{y}{\log y}\log \left( \frac{\log N}{y}\right)$$ as long as $y = o(\log N)$ and $y \to \infty$ as $N \to \infty$. 
In particular, when $y = (\log N)^\beta$ for $0 < \beta < 1$, 
$$\log \left(\frac{\log N}{y}\right) = \frac{1 - \beta}{\beta} \log y,$$ so 
$$\log (M_N(y, \phi)/N) \leq  (1 + o(1))\frac{1 - \beta}{\beta} y.$$
Since $s(N) \asymp N$, this proves Theorem \ref{thmthr}.

\section{Number of Forms with Degree $d$ Hecke Fields}\label{easysection}

For a prime level $N$, let $H(N)_d \subset H(N)$ denote Hecke forms whose Hecke eigenvalues span a number field of degree exactly $d$. We bound the size of $H(N)_d$ using the multiplicity bound from the previous section. 

Specifically, let $y > 0$, $r = \pi(y)$, and for $f \in H(N)_d$, and let $a_f(p) = \lambda_f(p) \sqrt{p}$ be the $p^\text{th}$ Hecke operator eigenvalue of $f$. To prove Theorem \ref{thmfour}, we combine the multiplicity bound with an upper bound on the set $$T_N(y)_d: = \{(a_f(p_1), \ldots, a_f(p_r)) | f \in H(N)_d  \}$$ of possible tuples of eigenvalues of a Hecke form at the first $r$ primes. We get this bound by exploiting that $a_f(p)$ is a totally real algebraic integers whose conjugates are bounded by $2 \sqrt{p}$ in size. 

\begin{prop}\label{tuplebound}

\begin{align}
\#\abs{T_N(y)_d} \leq \exp\left(y d/2 + o_d(y) \right).
\end{align} 
\end{prop}

\begin{lemma}\label{firstl}

For $f \in H(N)_d$, let $K_{f, r}:= \Q(a_f(p_1), \ldots, a_f(p_r))$. Then:

$$\#\abs{\{K_{f, r} | f \in H(N)_d\}} \ll_d y^{\kappa} $$ where $\kappa = \kappa(d)$ is a constant depending on $d$. 
\end{lemma}

\begin{proof}
Let $K = K_{f, r}$ for some $f \in H(N)_d$. Let $K_i : = \Q(a_f(p_i))$ be of degree $d_i \leq d$ with discriminant $\Delta_i$, and let $P_i(x) = \prod (x - \beta_j)$ be the minimal polynomial of $a_f(p_i)$. Then 
$$\abs{\Delta_i} = \frac{\abs{\mathrm{disc} (P_i)}}{[\mathcal{O}_{K_i} : \Z[a_f(p_i)]]^2} \leq \prod_{i \neq j}\abs{(\beta_i - \beta_j)} \leq (4 \sqrt{p_i})^{d_i(d_i - 1)} \ll_d y^{d^2/2}.$$ Since $K$ has degree at most $d$, it can be expressed as a composition of at most $\log_2 d$ fields $K_i$, so the discriminant $\Delta$ of $K$ satisfies $$\abs{\Delta} \ll_d y^k$$ for some constant $k$ depending only on $d$. This implies a bound of the same form on the number of possibilities for $K$ by the Theorem of Schmidt (\cite{shm}).
\end{proof}

\begin{lemma}\label{secondl}
Let $K$ be a totally real number field of degree $\leq d$. Then for $M > 1$, the number of $\alpha \in \mathcal{O}_K$ such that all the Galois conjugates of $\alpha$ are bounded by $M$ is at most $C(d) M^d$ for some constant $C(d)$ which does not depend on $K$. 
\end{lemma}
\begin{proof}
Consider that standard embedding $\iota: K \hookrightarrow \R^d$.  For $\alpha \in \mathcal{O}_K$, the coordinates of $\iota(\alpha)$ are the Galois conjugates of $\alpha$; their product is a non-zero integer, so the non-zero vectors in the lattice formed by the image of $\mathcal{O}_K$ under $\iota$ have length $\geq 1$. From this, sphere packing bounds imply immediately that the number of lattice points in the box $[-M, M]^d$ is bounded by $O_d(1) M^d$ (this can be seen, for example, by placing (disjoint) balls of diameter $1$ at each lattice point in the box and comparing volumes).
\end{proof}

\begin{proof}[Proof of Proposition \ref{tuplebound}]
From Lemma \ref{secondl}, we see that for a fixed degree $K$ number field, the number of possible tuples $(a_f(p_1), \ldots, a_f(p_r))$ with $a_f(p_i) \in K$ is at most
$$\prod_{p \leq y} C(d) (2\sqrt{p})^d = (2C(d))^r \exp((d/2) \sum_{p \leq y} \log p ) =  \exp(dy/2 + o_d(y)),$$ where the last step uses the prime number theorem. On the other hand, from Lemma \ref{firstl}, the number of choices for $K$ is $\exp(O_d(\log y)) = \exp(o_d(y))$, so multiplying the two proves the proposition statement.
\end{proof}
 Combining this Proposition with Theorem \ref{thmthr},
$$\log s(N)_d/s(N) \leq -\left(\frac{1 - \beta }{\beta}- \frac{d}{2}\right)y +  o_d(y),$$ which concludes the proof of Theorem \ref{thmfour}.

Note that for the coefficient of $y$ to be negative, we have to choose $\beta$ small, which is why we dealt with multiplicity bounds in Theorem $3$ only for $0 \leq \beta \leq 1$. 

\section{Composite Level}\label{lastsection}
The discussion up to this point was restricted to weight $k = 2$ and level $N$ being prime. Theorems \ref{thmone} and \ref{thmtwo} can be extended without much change to allow varying weight and general $N$, as long as the relatively prime conditions $(mn, N) = 1$ and $(p_j, N) = 1$ are maintained. Indeed, the starting point, which is an application of the Petersson formula (\cite{ILS}, Corollary $2.2$ or \cite{IK}, Corollary $14.24$), gives the desired uniformity in the "harmonic" weighted form. To remove the weights, one has to take care with new and old forms and the Atkin-Lehner involutions in relating $\norm{f}^2_2$ and $L(1, \sym^2 f)$. 

On the other hand, in the proof of the multiplicity bounds in Theorem \ref{thmthr}, we used $y$-smooth numbers and the assumption that $(p, N) = 1$ for $p \leq y < \log N$. As we show in Theorem \ref{thmfiv} below, similar bounds can be proved for $N$'s that do not have an abnormal number of small prime factors. For "super-smooth" numbers, such as $N = \prod_{p \leq t} p$, we cannot make use of the approach to the Plancherel measure of the Hecke eigenvalues for small primes, and our bounds in Theorem \ref{thmthr} and \ref{thmfour} don't apply. 

In what follows, we restrict ourselves to the $s^*(N)$-dimensional space $S^*(N)$ of weight $2$ level $N$ newforms, which admits a simultaneous eigenbasis $H^*(N)$ with respect to Hecke operators $T_n$ with $(n, N) = 1$ (we assume these forms are normalized to have constant Fourier coefficient $1$).

For a positive integer $N$, the number of distinct prime divisors of $N$ is at most 

\begin{align}\label{sizeofp}(1 + o(1)) \frac{\log N}{\log \log N} =: \bf{P}.
\end{align} 
Let $y= y(N) =  o(\log N)$ be a parameter going to infinity with $N$, and let $r := \pi(y) \sim y/\log y$. Let $q_1, \ldots, q_r$ be the first $r$ primes which don't divide $N$. Since $q_k$ is at most the $(k + \mathbf{P})^\text{th}$ prime and $k \leq r = o(\bf{P}),$
we can conclude via the prime number theorem that
\begin{align}\label{sizeofq}
q_k \leq (1 + o(1))(\bf{P} \log \bf{P}).
\end{align}
In the spirit of section \ref{multipli}, we want to give a lower bound for the function
$$\Phi(q_1, \ldots,  q_r, X) := \#\abs{\{(\alpha_1, \ldots, \alpha_r): q_1^{\alpha_1} \cdots q_r^{\alpha_r} \leq X\}}$$ for $X$ (to be chosen later) satisfying 
\begin{align}\label{xislogn}
\log X = (1 + o(1)) \log N.
\end{align}
By (\ref{sizeofq}),
\begin{align}\label{ineq}
\Phi(q_1, \ldots,  q_r, X) &= \#\abs{\{(\alpha_1, \ldots, \alpha_r) :  q_1^{\alpha_1} \cdots q_r^{\alpha_r} \leq X\}}\geq  \notag \\ 
&\geq \#\abs{\{(\alpha_1, \ldots, \alpha_r) :  \big((1 + o(1))\mathbf{P}\log \mathbf{P}\big)^{\alpha_1 + \cdots + \alpha_r} \leq X\}} =  \notag \\
&= \#\abs{\{(\alpha_1, \ldots, \alpha_r) :  \alpha_1 + \cdots + \alpha_r \leq \log X/\log \big((1 + o(1))(\mathbf{P} \log \mathbf{P}) \big)\}} \notag \overset{\tiny{\text{ by }}(\ref{sizeofp}),(\ref{xislogn})}{=}\\
&= \#\abs{\{(\alpha_1, \ldots, \alpha_r) :  \alpha_1 + \cdots + \alpha_r \leq (1 + o(1)) \log X/\log \log X\}} .
\end{align}
The number of non-negative integer solutions to $x_1 + \ldots + x_A \leq B$ is $${A + B \choose A} \geq \left(\frac{B}{A} \right)^A ,$$ so (\ref{ineq}) implies 
\begin{align}\label{sizeofohi}
\log \Phi(q_1, \ldots,  q_r, X) \geq r \log \frac{(1 + o(1)) \log X}{r \log \log X}.
\end{align}

For $\phi \in S^*(N)$ a weight $2$ holomorphic cusp newform for $\Gamma_0(N)$, we let $M^*_N(q_1, \ldots, q_r, \phi)$ be the multiplicity of the tuple of eigenvalues of $\phi$ at primes $q_i$, i.e. 
$$M^*_N(q_1, \ldots, q_r, \phi) := \#\abs{ \left\{ f\in H^*(N) : \lambda_f(q_i) = \lambda_\phi(q_i) \text{ for all } i \leq r\right \}}.$$ We bound $M^*_N(q_1, \ldots, q_r, \phi)$ for a fixed $\phi$ via the large sieve inequality identically to section \ref{multipli}. Taking 
\begin{align*}
c_n :=
\begin{cases} \overline{\lambda_\phi(n)} &\text{ if } n= q_1^{\alpha_1} \cdots q_r^{\alpha_r} \leq X \\
0 &\text{ otherwise,}
\end{cases}
\end{align*}
we get
\begin{align}\label{sizeofmm}
M^*_N(q_1, \ldots, q_r, \phi)  \ll N(\log N)^2/ \sum_{\substack{ n = q_1^{\alpha_1} \cdots q_r^{\alpha_r} \\ n \leq N}} \abs{\lambda_\phi(n)}^2.
\end{align}
Recreating the proof in section \ref{multipli}, we can see that 
\begin{align}\label{totooo}
\sum_{\substack{n \leq N\\ n \in \S_{q_1, \ldots, q_r}}} \abs{\lambda_\phi(n)}^2 \gg 8^{-r} \Phi\left(q_1, \ldots, q_r, \frac{N}{q_1 \cdots q_r}\right) \geq 8^{-r} \Phi\left(q_1, \ldots, q_r, \frac{N}{\mathbf{P}^{2r}}\right).
\end{align}
Let $X := N/\mathbf{P}^{2r}$. Recall that $r = \pi(o(\log N)) = o(\log N/\log \log N),$ so 
$$\log X = \log N  - 2r  (1 + o(1))\log \log N = \log N(1 + o(1)),$$ which means this choice of $X$ satisfies (\ref{xislogn}) and hence also satisfies (\ref{sizeofohi}).
Let $0 < \beta < 1$ and let
 $$y :=(\log N)^\beta,$$ so $$r = (1 + o(1)) (\log X)^\beta/\beta \log \log X.$$ Then (\ref{sizeofohi}) becomes

$$\log \Phi(q_1, \ldots, q_r, X) \geq (1 + o(1))\frac{ (\log X)^\beta}{\beta \log \log X} \log \frac{\beta \log X}{(\log X)^\beta} = (1 + o(1))\frac{ 1 - \beta}{\beta} (\log X)^\beta,$$ i.e.,

$$\log \Phi(q_1, \ldots, q_r, N) \geq (1 + o(1)) \frac{1 - \beta}{\beta} (\log N)^\beta.$$
Finally, by (\ref{sizeofmm}) and (\ref{totooo}),
\begin{align}\label{helo}
\log M^*_N(q_1, \ldots, q_r, \phi)/N \leq \log \log N + r \log 8 - (1 + o(1))\frac{1 - \beta}{\beta} (\log N)^\beta = - (1 + o(1))\frac{1 - \beta}{\beta} (\log N)^\beta
\end{align} (note that this bound is identical to the one in section \ref{multipli} which is sharp).

We apply (\ref{helo}) to extend Theorem \ref{thmfour} to more general $N$'s. For $T \geq 1$ fixed and for some $y = y(N)$ with $\log \log N \ll y \ll \log N$, we say that a large $N$ is $T$-super-smooth if 
$$\frac{\pi(y^T; N) }{ \pi(y)} = o(1),$$ where $\pi(z; N) = \# \abs{\{p \leq z: (p, N) = 1 \}}$ is the number of primes up to $z$ that don't divide $N$. Clearly, very few numbers are $T$-super-smooth for all $T$. 

Let $H^*(N)_d := \{f \in H^*(N) : d(f) = d\}$ be the set of Hecke newforms whose Fourier coefficients span a number field of degree $d$, $s^*(N)_d = \abs{H^*(N)_d}$. The following theorem extends Theorem \ref{thmfour} to non-super-smooth numbers.

\begin{thm}\label{thmfiv}
Let $0 \leq \beta \leq 1$, $y = (\log N)^\beta$, $T \geq 1$, and $d \geq 1$. Then for $N$ not $T$-super-smooth,

$$ s^*(N)_d \leq \exp\left(-  \left(\frac{1 - \beta}{\beta} - \frac{dT}{2}\right)y + o_{T, d}(y)\right) s^*(N)$$ as $N \to \infty$. 
\end{thm}
\begin{proof}
The proof emulates that of section \ref{easysection}. For $f \in H^*(N)_d$, let $a_f(q_i) = \lambda_f(q_i) \sqrt{q_i}$ be the eigenvalue of $f$ for the Hecke operator $T_{q_i}$, and let $$T^*_N(q_1, \ldots, q_r)_d: = \{ (a_f(q_1), a_f(q_2), \ldots, a_f(q_r)) | f \in H^*(N)_d\}$$ denote the set of possible eigenvalue tuples of a form in $H^*(N)_d$ at the first $r$ primes not dividing $N$. Repeating verbatim the proof of Lemma \ref{firstl}, there are $\ll_d y^{T\kappa(d)}$ possible number fields of the form $\Q(a_f(q_1), \ldots, a_f(q_r))$. Tautologically, for $N$ as in the statement of the theorem, the first $r$ primes $q_1, \ldots, q_r$ not dividing $N$ satisfy $q_i \leq y^{T}$, so using Lemma \ref{secondl}, we get 
$$  \# \abs{T^*_N(q_1, \ldots, q_r)} \leq y^{T \kappa(d)} \prod_{i \leq r} C(d) y^{Td/2}   \leq \exp((d/2)T r \log y + o_{T,d}(y)) \leq \exp((d/2)T y + o_{d,T}(y)).$$ Combined with (\ref{helo}), this gives 

$$\log s(N)_d/ N \leq -  \left(\frac{1 - \beta}{\beta} - \frac{d T}{2}\right)y   + o_{T, d}(y).$$ It remains to note that $s^*(N) \asymp \phi(N)$, the Euler's totient function, and $\log \phi(N) = \log N + O(\log \log \log N),$ which completes the proof.

\end{proof}

\bibliographystyle{alpha}
\bibliography{szpaper}

\begin{thebibliography}{CDF97}

\bibitem[CDF97]{CDF}
J.~B. Conrey, W.~Duke, and D.~W. Farmer.
\newblock The distribution of the eigenvalues of {H}ecke operators.
\newblock {\em Acta Arith.}, 78(4):405--409, 1997.

\bibitem[CGL09]{cgl}
Denis~X. Charles, Eyal~Z. Goren, and Kristin~E. Lauter.
\newblock Families of {R}amanujan graphs and quaternion algebras.
\newblock In {\em Groups and symmetries}, volume~47 of {\em CRM Proc. Lecture
  Notes}, pages 53--80. Amer. Math. Soc., Providence, RI, 2009.

\bibitem[DK00]{DK}
W.~Duke and E.~Kowalski.
\newblock A problem of {L}innik for elliptic curves and mean-value estimates
  for automorphic representations.
\newblock {\em Invent. Math.}, 139(1):1--39, 2000.
\newblock With an appendix by Dinakar Ramakrishnan.

\bibitem[Eic54]{Eichlerramanujan}
Martin Eichler.
\newblock Quatern\"{a}re quadratische {F}ormen und die {R}iemannsche
  {V}ermutung f\"{u}r die {K}ongruenzzetafunktion.
\newblock {\em Arch. Math.}, 5:355--366, 1954.

\bibitem[HL94]{HL}
Jeffrey Hoffstein and Paul Lockhart.
\newblock Coefficients of {M}aass forms and the {S}iegel zero.
\newblock {\em Ann. of Math. (2)}, 140(1):161--181, 1994.
\newblock With an appendix by Dorian Goldfeld, Hoffstein and Daniel Lieman.

\bibitem[HT86]{hild}
Adolf Hildebrand and G\'{e}rald Tenenbaum.
\newblock On integers free of large prime factors.
\newblock {\em Trans. Amer. Math. Soc.}, 296(1):265--290, 1986.

\bibitem[IK04]{IK}
Henryk Iwaniec and Emmanuel Kowalski.
\newblock {\em Analytic number theory}, volume~53 of {\em American Mathematical
  Society Colloquium Publications}.
\newblock American Mathematical Society, Providence, RI, 2004.

\bibitem[ILS00]{ILS}
Henryk Iwaniec, Wenzhi Luo, and Peter Sarnak.
\newblock Low lying zeros of families of {$L$}-functions.
\newblock {\em Inst. Hautes \'{E}tudes Sci. Publ. Math.}, (91):55--131 (2001),
  2000.

\bibitem[IM01]{IM}
H.~Iwaniec and P.~Michel.
\newblock The second moment of the symmetric square {$L$}-functions.
\newblock {\em Ann. Acad. Sci. Fenn. Math.}, 26(2):465--482, 2001.

\bibitem[Iwa84]{crelle}
Henryk Iwaniec.
\newblock Prime geodesic theorem.
\newblock {\em J. Reine Angew. Math.}, 349:136--159, 1984.

\bibitem[Lin41]{linnik}
U.~V. Linnik.
\newblock The large sieve.
\newblock {\em C. R. (Doklady) Acad. Sci. URSS (N.S.)}, 30:292--294, 1941.

\bibitem[LW08]{LauWu}
Y.-K. Lau and J.~Wu.
\newblock A large sieve inequality of {E}lliott-{M}ontgomery-{V}aughan type for
  automorphic forms and two applications.
\newblock {\em Int. Math. Res. Not. IMRN}, (5):Art. ID rnm 162, 35, 2008.

\bibitem[MS09]{MS}
M.~Ram Murty and Kaneenika Sinha.
\newblock Effective equidistribution of eigenvalues of {H}ecke operators.
\newblock {\em J. Number Theory}, 129(3):681--714, 2009.

\bibitem[NS21]{ns}
Evita Nestoridi and Peter Sarnak.
\newblock Bounded cutoff window for the non-backtracking random walk on
  ramanujan graphs.
\newblock {\em arXiv preprint arXiv:2103.15176}, 2021.

\bibitem[Sar87]{sar1}
Peter Sarnak.
\newblock Statistical properties of eigenvalues of the {H}ecke operators.
\newblock In {\em Analytic number theory and {D}iophantine problems
  ({S}tillwater, {OK}, 1984)}, volume~70 of {\em Progr. Math.}, pages 321--331.
  Birkh\"{a}user Boston, Boston, MA, 1987.

\bibitem[Sar02]{rudnickletter}
Peter Sarnak.
\newblock Letter to {Z}. {R}udnick on multiplicities of eigenvalues for the
  modular surface.
\newblock 2002.

\bibitem[Sch95]{shm}
Wolfgang~M. Schmidt.
\newblock Number fields of given degree and bounded discriminant.
\newblock Number 228, pages 4, 189--195. 1995.
\newblock Columbia University Number Theory Seminar (New York, 1992).

\bibitem[Ser97]{serre}
Jean-Pierre Serre.
\newblock R\'{e}partition asymptotique des valeurs propres de l'op\'{e}rateur
  de {H}ecke {$T_p$}.
\newblock {\em J. Amer. Math. Soc.}, 10(1):75--102, 1997.

\bibitem[Wei48]{weil}
Andr\'{e} Weil.
\newblock On some exponential sums.
\newblock {\em Proc. Nat. Acad. Sci. U.S.A.}, 34:204--207, 1948.

\end{thebibliography}

\end{document}